\documentclass[12pt, reqno]{amsart}

\usepackage[utf8]{inputenc}
\usepackage[T1]{fontenc}
\usepackage{tikz,tikz-3dplot}
\usepackage{amsmath,amsfonts,amssymb,amsmath,latexsym,mathrsfs}
\usepackage{shuffle}
\usepackage[all]{xy}
\usepackage{stackengine}
 
\usepackage{enumerate}

\usepackage{hyperref}
 
\usepackage{tabu}
\usepackage{tikz-cd} 

\usepackage{comment}
\usepackage{url}
\usepackage{tikz}

\usepackage{xfrac}

\newtheorem{theorem}{Theorem}[section]

\newtheorem{proposition}[theorem]{Proposition}
\newtheorem{definition}[theorem]{Definition}

\newtheorem{corollary}[theorem]{Corollary}

\newtheorem{remark}[theorem]{Remark}

\newtheorem{example}[theorem]{Example}

\newcommand{\C}{ {\mathbb C} }

\begin{document}

\title{On irreducible representations of quandles}

\author{Mohamad \textsc{Maassarani} }
\maketitle
\begin{abstract}We consider irreducible representations of finite quandles over $\mathbb{C}$. For $Q$ a finite quandle whose inner automorphism group $Inn(Q)$ have trivial Schur multipliers, we prove that the irreducible representations of $Q$ can be constructed out of what we call characters of $Q$ and irreducible linear represenations of the group $Inn(Q)$. For $G$ a finite groiup having trivial Schur multiplier or being a Schur cover of another group, we show that the irreducible representations of the conjugacy quandle $Conj(G)$ can be constructed out of characters of $Conj(G)$ and irreducible linear representations of the group $G$. In both cases, the finite unitary irreducible representations can be determined from the results. For instance, these results allow to solve the problem of constucting irreducible represenations of the conjugacy quandles of dihedral groups and generalised quaternion groups. In general, we relate the irreducible representations of a finite quandle $Q$ to irreducible projective representations of $Inn(Q)$ and prove that the irreducible representations of $Q$ can be in theory constructed out of characters of $Q$ and irreducible representations of a finite quotient of the enveloping group $G(Q)$. The quotient is a stem extensions of $Inn(Q)$ with nucleus a finite subgroup of the center of $G(Q)$. This allows, using a result from the litterature, to show that the irreducible quandle representations of $Conj(S_n)$ ($S_n$ the symmetric group) can be constructed out of characters of the corresponding quandle and irreducible linear group representations of the symmetric group. 
\end{abstract}
\section*{Introduction and main results}
These notes are related to represenation theory over $\mathbb{C}$ of finite quandles. The basic definitions are given in the first section. A quandle $Q$ is a set equipped with binary operation satisfying some axioms. To each quandle $Q$ one associates an inner automorphism group $Inn(Q)$ generated by translation. For a group $G$ there is an associated quandle $Conj(G)$ called the conjugacy quandle of the group. One can associate to a quandle a group $G(Q)$ and a quandle morphism $Q \to G(Q)$ that extends any quandle morphism $Q\to Conj(G)$ into a group morphism $G(Q)\to G$. This group is known as the enveloping group, structure group, associated group or adjoint group.  In \cite{rep}, the authors define a quandle representation of quandle $Q$ (more generally a rack representation) as a quandle morphism $Q\to Conj(GL(V))$. They also define a notion of strong represenations of a quandle.  In \cite{str}, the authors show that there is a correspondance between strong representations of a finite quandle and linear representations of its inner automorphism group. They also consider irreducible strong representations. In \cite{GQ}, we prove that irreducible representations (not necesseraly strong) of finite quandles over $\mathbb{C}$ are finite dimensional. In these notes we consider irreducible representations (not necessarly strong) of finite quandles over $\C$.
\subsection*{Outline of the paper :}\ \\\\
In section $1$, we recall some definitions and basic properties concerning: quandles, inner automorphisms of quandles, conjugacy quandle of a group, envelopping groups and quandle representations. We also remind some elements on projective representations.\\\\
In section $2$, we establish the correpondance between irreducible representations of a quandle $Q$ and irreducible representations of its envelopping group $G(Q)$. We definie characters of a quandle wich correpond to quandle morphisms to $\mathbb{C}^\times$. We show that an irreducible representation of a finite quandle $Q$ induces an irreducible projective representation of its inner automorphism group $Inn(Q)$. We show that two irreducible representations of a finite quandle inducing the same projective representation differ by a character. We use this to prove that :
\begin{itemize}
\item[-] if the Schur multiplier of $Inn(Q)$ ($Q$ finite) is trivial then the irreducible representations of $Q$ can be constructed out of characters of $Q$ and linear representations of $Inn(Q)$.
\item[-] if a group $G$ ($G$ finite) is a Schur cover of another group or has trivial Schur multiplier then, the irreducible representations of $Conj(G)$ can be constructed of characters of $Conj(G)$ and linear irreducible representations of the group $G$.
\end{itemize}
In both cases the unitary (taking values in the unitary group) irreducible quandle representations can be determined. The second case provide examples where irreducible projective representations of $Inn(Q)$ that do not lift to linear representations of $Inn(Q)$ are induced by irreducible quandle representations. \\\\
In section 3, we show that in theory the irreducible representations of a finite quandle $Q$ can be constructed out of characters and irreducible representations of a finite quotient of the envelopping group $G(Q)$. One can distinguish the irreducible quandle representations taking values in the unitary group. We also prove that a finite quandle $Q$ defines a subgroup $M_Q$ of the Schur multiplier of $Inn(Q)$ corresponding to classes of irreducible projective representations of $Inn(Q)$ induced by $Q$. $M_Q$ is related to a torsion subgroup of the center of $G(Q)$. We determine $M_Q$ for some examples. This allows to compute the torsion subgroup of the center of $G(Q)$ for these examples. This also allows to prove that the irreducible quandle representations of the conjugacy quandle of the Symmetric group can be constructed out of characters of the quandle and irreducible linear group representations of the symmetric group. For that we use a result from the litterature asserting that the enveloping group of the conjugacy quandle of the symmetric group has no torsion (\cite{leb}). 
\section{Quandles, inner automorphisms, conjugacy quandles, envelopping groups and quandle representations}

\subsection{Quandles, quandle morphisms}
\begin{definition}
 A quandle is a set $Q$ equipped with a binary operation $\triangleright:Q\times Q \to Q$  satisfying the following :
\begin{itemize}
\item[-] $x\triangleright x=x$ for $x\in Q$.
\item[-] $x \triangleright( y\triangleright z)= (x \triangleright y) \triangleright (x \triangleright z)$, for $x,y,z \in Q$. 
\item[-] For all $x,y \in Q$ there exist a unique $z\in Q$, such that $x \triangleright z=y$.
\end{itemize}
\end{definition}
\begin{definition}
 A quandle morphism is a map between two quandles $(Q,\triangleright_Q)$ and $(Q' ,\triangleright_{Q'})$ such that $f(z \triangleright_Q w)=f(z) \triangleright_{Q'} f(w)$, for $z,w \in Q$.
\end{definition}
The composition of two quandle morphism is a quandle morphism.
\begin{definition}
A subquandle of $(Q,\triangleright)$ is a subset $Q'$ such that $\triangleright$ restricts to quandle structure on $Q'$.

\end{definition} The image of a quandle morphism is a subquandle.  
\subsection{Inner automorphisms of quandles}
For $Q$ a quandle it follows from the definition of a quandle we took that for $x\in Q$  the map $L_x: Q\to Q$ given by $L_x(y)=x\triangleright y$ is a bijective quandle morphism. 
\begin{definition}
For $Q$ a quandle, the group of inner automorphism $Inn(Q)$  of $Q$  is the subgroup of the group of bijections of $Q$ generated by the maps $L_x$ for $x\in Q$.
\end{definition}
\begin{proposition}
The map $Q\to Inn(Q), x\mapsto L_x$ is a quandle morphism if $Inn(Q)$ is endowed with the conjugacy quandle structure.
\end{proposition}
\begin{definition}
The orbits of a quandle $Q$ are the orbits of $Q$ under the action of $Inn(Q)$.
\end{definition}
\subsection{Conjugacy quandle of a group}
\begin{definition}
For a group $G$, the conjugacy quandle $Conj(G)$ of $G$ is a quandle consisting of the set $G$ equipped with the binary operation $x\triangleright y= xyx^{-1}$  for $x,y\in G$.
\end{definition}
A group morphism $f:G\to H$  is a quandle morphism with respect to the conjugacy quandle strucutre just defined. 
\begin{proposition}
For $G$ a group, the group $Inn(Conj(G))$ is equal to $Inn(G)\simeq G/Z(G)$ where $Z(G)$ is the center of $G$ and the orbits of $Inn(Conj(G))$ are the conjugacy classes of $G$.
\end{proposition}


\subsection{Quandle representations}
\begin{definition}
\item[1)] A linear representation of a quandle $Q$ over a vector space $V$ is a quandle morphism $\rho: Q \to Conj(\mathrm{GL}(V))$, i.e. : 
$$\rho(x\triangleright y)=\rho(x) \rho(y) \rho(x)^{-1},$$
for all $x,y \in Q$.
\item[2)] A subrepresentation of $\rho$ is a vector subspace $W\subset V$, such that $\rho(x)(W) \subset W$ for all $x\in Q$.
\item[3)] An irreducible representation is a representation $V$ that has no subrepresentations other than $0$ and $V$.
\end{definition} 
We define intertwining operators and equivalence of quandle representations as for groups. Note that if $f:Q \to Conj(G)$ is a quandle morphism and $\rho : G \to GL(V)$ is a group represenation of $G$, then $\rho\circ f$ is a quandle representation.
\subsection{Envelopping group}

\begin{definition}
The enveloping group $G(Q)$ of a quandle $Q$ is the group definied by generators and relations :
$$G(Q)=\langle x \in Q \vert xyx^{-1}=x\triangleright y , \text{ for } x,y\in Q\rangle.$$
\end{definition}
The enveloping group is also called : the associated group, adjoint group or structure group. The group $G(Q)$ is infinite since word length (length of a geneartor is $1$ and the inverse of such element has lenght $-1$) induces a non trivial morphism $G(Q)\to \mathbb{Z}$. 
\begin{proposition}
 
\item[1)] The map $\varphi_Q : Q \to G(Q)$ mapping $x\in Q$ to the corresponding generator is a quandle morphism if we endow $G(Q)$ with the conjugacy quandle structure.
\item[2)]If $f:Q\to Conj(G)$ is a quandle morphism then there is a unique group morphism $f_{G(Q)}:Q\to G$ such that the following diagram commutes :
$$\begin{tikzcd}
G(Q) \arrow{r}{f_{G(Q)}} &G  \\ 
Q \arrow{u}{\varphi_Q} \arrow{ru}{f} & 
\end{tikzcd} $$
 
\end{proposition}
\subsection{Projective representations}
 In this subsection the groups we consider are \textbf{finite} and the vector spaces are \textbf{complex}.
\begin{definition}
\item[1)]  A projective representation of a group $G$ is a morphism $G\to PGL(V)$.
\item[2)]  A projective representation $G \to PGL(V)$ is irreducible if $P(V)$ admit no proper projective subspace stable by $G$.
\item[3)] Two projective representations are projectively equivalent if they are conjugate one to another by an invertible projective transformation.
\end{definition} 
Let $\rho : G \to PGL(V)$ be a projective representation. Let $\tilde{\rho}$ be a lift of $\rho$ to $GL(V)$ such that $\tilde{\rho}(1)=Id$. For $g,h\in G$ :
$$\tilde{\rho}(gh)=\alpha_{\tilde{\rho}}(g,h)\tilde{\rho}(g)\tilde{\rho}(h),$$
for some $\alpha_{\tilde{\rho}}(g,h)\in \mathbb{C}^\times$. One can check that the mapping $\alpha_{\tilde{\rho}} :G\times G\to \mathbb{C}^\times$ corresponding to the constants is a $2$-cocycle with values in $\mathbb{C^\times}$. Let $\tilde{\rho}'$ be another lift of $\rho$ to $GL(V)$ such that $\tilde{\rho}'(1)=Id$. The cocycles $\alpha_{\tilde{\rho}}$ and $\alpha_{\tilde{\rho}'}$ differ by a boundry. Hence, the projective representation $\rho$ definies a cohomology class $[\rho]$ in the group $H^2(G,\mathbb{C}^\times)$ called the \textbf{Schur multiplier} of $G$. The Schur multiplier of a finite group is finite.
\begin{proposition}
\item[1)] A projcetive representation $\rho :G \to PGL(V)$ lifts to a linear representation $G\to GL(V)$ if and only if its class in $H^2(G,\mathbb{C}^\times)$ is trivial.
\item[2)] Two projectively equivalent representations definie the same cohomology class in $H^2(G,\mathbb{C}^\times)$.
\end{proposition}
One can use twisted algebras to prove the following (\cite{Sc} p. 40-42) :
\begin{proposition}
One can associate to each class $\alpha$ of the Schur multiplier $H^2(G,\mathbb{C}^\times)$ at least one irreducible projective representation of $G$ over $\mathbb{C}$ having $\alpha$ as a class.
\end{proposition}
Let $\theta: G\to H$ be a group morphism.  To a $2$-cocyle $\alpha$ of $H$ with values in $\mathbb{C}^\times$ one can assign the $2$-cocycle $\alpha_{\theta}=\alpha \circ (\theta \times \theta)$ of $G$ with values in $\mathbb{C}^\times $. This assignement induces a morphism $H^2(H,\mathbb{C}^\times)\to H^2(G,\mathbb{C}^\times)$ called the inflation map and will be denoted by $Inf_\theta$. 
\begin{proposition}
Let $\rho:H\to PGL(V)$ be a projective representation and $\theta : G\to H$ be a group morphism. $\rho$ lifts to a linear representation of $G$ via $\theta$, $($i.e. $\rho \circ \theta$ lifts to a linear representation of $G$$)$ if and only if the class of $\rho$ in $H^2(H,\mathbb{C}^\times)$ lies in the kernel of the inflation map $Inf_\theta$. 
\end{proposition}
\begin{proposition}
Let $1\to A \to G \overset{\theta}{\to} H $ be a central extention. We have an exact sequence $($inflation-restrection exact sequence with coefficient in $\mathbb{C}^\times$$)$ :
$$1 \to  Hom(H,\mathbb{C}^\times)\to Hom(G,\mathbb{C}^\times)\overset{Res}{\to} Hom(A,\mathbb{C}^\times)\to H^2(H,\mathbb{C}^\times)\overset{Inf_\theta}{\to}H^2(G,\mathbb{C}^\times),$$
where $Res$ is the restriction morphism.
\end{proposition}
\begin{definition}
A central extension $1\to A \to G \overset{\theta}{\to} H $ is called a stem extension if the nucleus $A$ seen as a subgroup of $G$ lies in the derived subgroup of $G$.
\end{definition}
For a stem extension $1\to A \to G \overset{\theta}{\to} H $ the inflation-restriction exact sequence gives an exact sequence :
$$0 \to Hom(A,\mathbb{C}^\times)\to H^2(H,\mathbb{C}^\times)\overset{Inf_\theta}{\to}H^2(G,\mathbb{C}^\times).$$
Hence, $A$ is isomorphic to a subgroup of the Shur multiplier for a stem extension.
\begin{definition}
A group $G$ is a Schur cover of $H$, if we have a morphism $\theta :G \to H$, such that $1\to ker(\theta) \to G \overset{\theta}{\to} H\to 1$ is a stem extension with $ker(\theta) \simeq H^2(H,\mathbb{C}^\times)$.
\end{definition}
Hence, if $G$ is a Schur cover of $H$ associated to the morphism $\theta: G\to H$, the kernel of $Inf_\theta$ is the whole group $H^2(H,\mathbb{C}^\times)$. In particular : 
\begin{proposition}
If $G$ is a Schur cover of $H$ associated to the morphism $\theta: G\to H$, then any projective irreducible representation $\rho$ of $H$ lifts via $\theta$ to an irreducible linear representation of $G$ $($i.e. $\rho \circ \theta$ lifts to a linear representation of $G$$)$. 
\end{proposition}
\section{irreducible representations of quandles}
Let $Q$ be a quandle and $\rho : Q \to GL(V)$ be a quandle representation. By the universal property of $\varphi_Q :Q\to G(Q)$ there is a unique group representation $\rho_{G(Q)}:G(Q)\to GL(V)$ such that the following diagram commutes :
$$\begin{tikzcd}
G(Q) \arrow{r}{\rho_{G(Q)}} &GL(V)  \\ 
Q \arrow{u}{\varphi_Q} \arrow{ru}{\rho} & 
\end{tikzcd} .$$ 
For $\rho$ a represenation of $Q$ we will use the notation $\rho_{G(Q)}$ for the corresponding represenation of $G(Q)$.
\begin{proposition}
 
\item[1)] We have a one to one correpondance between quandle represenations of $Q$ and group representations of $G(Q)$ given by $\rho \mapsto \rho_{G(Q)}$.
\item[2)] Irreducible representations of $Q$ correspond to irreducible represenations of $G(Q)$ via this correspondance.
 
\end{proposition}
\begin{proof}
$1)$ follows from the universal property and the fact that a group represenation of $G(Q)\to GL(V)$ composed by $\varphi_Q$ gives a quandle represenation $Q\to GL(V)$. $2)$ is true since the image of $\varphi_Q$ generates $G(Q)$ and hence a subrepresentations of $\rho$ correpond to subrepresentations of $\rho_{G(Q)}$.
\end{proof}
\begin{definition}
A character of a quandle $Q$ is a quandle morphism $\chi : Q \to \mathbb{C}^\times$, i.e. :$$\chi(x\triangleright y) =\chi(y),$$ for $x,y\in Q$.
\end{definition}
\begin{proposition}
\item[1)] By the universal property of $G(Q)$ there is a one to one correpondance between characters of $Q$ and multiplicative characters of $G(Q)$.
\item[2)] Characters of a quandle $Q$ are the maps to $\mathbb{C}^{\times}$ constant on the orbits of $Q$ with respect to the action of $Inn(Q)$.
\end{proposition}

Denote by $\theta : Q \to Inn(Q)$ the map associating to $x\in Q$ the left translation $L_x$ defined by $L_x(y)=x\triangleright y$. $\theta$ is a quandle morphism with $Inn(Q)$ endowed with the conjugacy structure. Hence there is a unique group morphism $\theta_{G(Q)}:G(Q) \to Inn(Q)$ such that the following diagram commutes :
$$\begin{tikzcd}
G(Q) \arrow{r}{\theta_{G(Q)}} &Inn(Q)  \\ 
Q \arrow{u}{\varphi_Q} \arrow{ru}{\theta} & 
\end{tikzcd} .$$ 
We will denote by $Z_0$ the kernel of $\theta_{G(Q)}$.
\begin{proposition}{\cite{EM}}
The kernel $Z_0$ of $\theta_{G(Q)}$ is a subgroup of the center of $G(Q)$.
\end{proposition}
\begin{remark}
If the map $\varphi_Q :Q\to G(Q)$ is injective then $Z_0$ is exactly the center of $G(Q)$ $($\cite{dar}$)$.
\end{remark}
\begin{example}\label{ex}
If $Q=Conj(G)$, $Z_0$ is the center of $G(Q)$. Indeed, $G\to G$ is quandle morphism and hence factors throught $\varphi_Q: Q \to G(Q)$. Hence $\varphi_Q$ is injective.
\end{example}
\begin{proposition}{\cite{GQ}}
For $Q$ a finite quandle, an irreducible representation of $Q$ $($respectively of $G(Q)$$)$ over $\mathbb{C}$ is finite dimensional.
\end{proposition}
From now on $Q$ is a \textbf{finite} quandle and $\rho : Q \to GL(V)$ is an \textbf{irreducible} repesentation \textbf{over} $\C$. In particular $V$ is finite dimensional. All the represenatations we will consider are \textbf{over} $\mathbb{C}$. Let $s$ be a section of $\theta_{G(Q)}$. Define $\rho_{Inn(Q)}: Inn(Q) \to PGL(V)$ by :  
$$\rho_{Inn(Q)}(a)=\overline{\rho_{G(Q)}(s(a))},$$
where $\overline{\rho_{G(Q)}(s(a))}$ is the class of $\rho_{G(Q)}(s(a))$ in $PGL(V)$ for $a\in Inn(Q)$. 
\begin{proposition}\label{pro}
\item[1)] $\rho_{Inn(Q)}$ is an irreducible projective representation of $Inn(Q)$ and does not depend on the choice of the section $s$.
\item[2)] For $x\in Q$, $\rho_{Inn(Q)}(L_x)=\overline{\rho(x)}$, where $\overline{\rho(x)}$ is the class of $\rho(x)$ in $PGL(V)$.
\item[3)] Let $\rho':Q \to GL(W)$ be a quandle represenation equivalent to $\rho$. Then $\rho'_{Inn(Q)}$ is projectively equivalent to $\rho_{Inn(Q)}$. In particular, if $\rho'_{Inn(Q)}$ and $\rho_{Inn(Q)}$ are not equivalent then $\rho$ and $\rho'$ are not equivalent.
\end{proposition}
\begin{proof}
$1)$ follows from the fact that $Z_0$ lies in the center of $G(Q)$ and hence acts by scalars in the representation $V$. To obtain $2)$ for a given $x$, we can take $s$ such that $s(L_x)=\varphi_Q(x)$. We can use $2)$ to prove $3)$, since translations generates $Inn(Q)$.
\end{proof}
We define the product $\chi \cdot \rho$ of $\rho$ by a character $\chi: Q\to Conj(\mathbb{C}^\times)$ by the following : 
$$ \chi \cdot \rho (x)=\chi(x) \rho(x),$$
for $x \in Q$. As one can check $\chi \cdot \rho$ is an irreducible representation of $Q$. Moreover it follows from $2)$ of the previous proposition and the fact that left translations generates $Inn(Q)$ by definition that :  $$(\chi \cdot \rho)_{Inn(Q)}=\rho_{Inn(Q)}.$$
\begin{proposition}\label{pro1}
\item[1)] Let $\rho':Q\to GL(V)$ be a representation such that $\rho'_{Inn(Q)}=\rho_{Inn(Q)}$. We have that $\rho'=\chi \cdot \rho$ for some character $\chi$ of $Q$.
\item[2)] Assume that $\rho_{Inn(Q)}$ lifts to a linear represenation $\tilde{\rho}_{Inn(Q)}:Inn(Q)\to GL(V)$, then $\tilde{\rho}_{Inn(Q)}\circ\theta$ is an irreducible represenation of $Q$ and $$\rho=\chi \cdot (\tilde{\rho}_{Inn(Q)}\circ\theta),$$ for some character $\chi$ of $Q$.
\end{proposition} 
\begin{proof}
We prove $1)$. If $\rho'_{Inn(Q)}=\rho_{Inn(Q)}$ then for $x,y \in Q$, we have :
$$ \rho'(x)=\chi(x)\rho(x), \quad \rho'(y)=\chi(y) \rho(y) \quad\text{ and } \rho'(x\triangleright y)=\chi(x\triangleright y) \rho(x\triangleright y),$$
for some $\chi(x),\chi(y),\chi(x\triangleright y)\in \mathbb{C}^\times$. Now :
$$\rho'(x\triangleright y)=\rho'(x)\rho'(y)\rho'(x)^{-1}=\chi(y)\rho(x)\rho(y)\rho(x)^{-1}=\chi(y) \rho(x\triangleright y).$$
This proves that $\chi(x\triangleright y)\rho(x\triangleright y)=\chi(y)\rho(x\triangleright y)$ and $1)$ follows.\\We prove $2)$. The representation $\tilde{\rho}_{Inn(Q)}$ is irreducible since it lifts an irreducible projective representation. The quandle representation $\tilde{\rho}_{Inn(Q)}\circ \theta$ is hence irreducible since the image of $\theta$ generates $Inn(Q)$. By $2)$ of the previous proposition, we have for $x$ in $Q$ :
$$(\tilde{\rho}_{Inn(Q)}\circ\theta)_{Inn(Q)}(L_x)=\overline{(\tilde{\rho}_{Inn(Q)}\circ\theta)(x)}=\overline{\tilde{\rho}_{Inn(Q)}(L_x)}=\rho_{Inn(Q)}(L_x),$$
where the overline is for the classe in $PGL(V)$. This proves that $(\tilde{\rho}_{Inn(Q)}\circ\theta)_{Inn(Q)}=\rho_{Inn(Q)}$ since left translation generate $Inn(Q)$. We can therefore apply $1)$ of this proposition and obtain $2)$ since the inverse of a character is a character.
\end{proof}
\begin{theorem}
Let $Q$ be a finite quandle. Assume that $Inn(Q)$ has trivial Shcur multiplier. The irreducible representations of $Q$ are the representations $\chi \cdot (\rho'\circ \theta),$ where $\rho'$ is an irreducible linear represenation of the group $Inn(Q)$ and $\chi$ is a character of $Q$.
\end{theorem}
\begin{proof}
Let $\rho'$ be an irreducible representation of $Inn(Q)$. The quandle representation $\rho'\circ \theta$ is irreducible since the image of $\theta$ generates the group $Inn(Q)$. It follows that $\chi \cdot (\rho'\circ \theta) $ is irreducible. Now let $\rho$ be an irreducible representation of $Q$. Since we assume that $Inn(Q)$ has trivial Schur multiplier then $\rho_{Inn(Q)}$ lifts to a linear representation $\tilde{\rho}_{Inn(Q)}$. The theorem hence follows from $2)$ of the previous proposition
\end{proof}
\begin{remark}
Considering the determinant one shows that the representations $\chi \cdot (\rho'\circ \theta)$ of the therorem take values in the unitary group if and only if $\rho'$ is unitary and $\chi$ take values in complex numbers with modulus $1$.
\end{remark}
\begin{example}
Let $G$ be one of the groups $D_{2n}$ $($dihedral group with $2n$ elements$)$,  $D_{4n}$ or $Q_{4n}$ $($the generalised quaternion group with $4n$ elements$)$ for $n$ odd. We have $Inn(Conj(G))=Inn(G)\simeq D_{2n}$. The group $D_{2n}$ for $n$ odd has trivial Schur multiplier and hence the irreducible quandle representations of $Conj(G)$ can be constructed out of the irreducible linear representations of the group $D_{2n}$.  
\end{example}

Let $G$ be a finite group and set $Q=Conj(G)$. One has a unique group morphism $\pi : G(Q)\to G$ extending the identity $Conj(G)\to Conj(G)$. The morphism $\theta_{G(Q)}:G(Q)\to Inn(Q)$ factors throught $\pi$. Therefore, $\pi$ is a central extension and an irreducible quandle representation of $Q$ induces an irreducible projective representation of $G$. Mimicking what we have done for the extension $G(Q)\to Inn(Q)$, we can prove that :
\begin{theorem}
 For $G$ a finite group, if $H^2(G,\mathbb{C}^\times)$ is trivial, then the irreducible quandle represenations of $Conj(G)$ are the product of quandle characters of $Conj(G)$ by irreducible linear group representations of $G$.
\end{theorem}
\begin{example}
The generalised quaternion group $Q_{4n}$ has trivial Schur multiplier.
\end{example}
\begin{theorem}
Let $G$ be a finite group who is a Schur cover of some other group. The irreducible quandle representations of $Conj(G)$ are the representations $\chi\cdot \rho'$, where $\rho'$ is an irreducible linear representation of the group $G$ and $\chi$ is a charachter of $Conj(G)$.
\end{theorem}

\begin{proof}
It is not hard to prove that the representations in the therorem are irreducible. We need to prove that every irreducible representation is of that form. To simplify the notations we will use $Inn(Q)$ instead of $Inn(Conj(G))$ in the notations. Assume $G$ is a Schur cover of $H$. In particular, $G$ is a central extension of $H$. Since $Inn(Q)=Inn(G)\simeq G/Z(G)$, where $Z(G)$ is the center of $G$, we have a commutative diagramm : 
$$\begin{tikzcd}
G \arrow{d}{\alpha}\arrow{r}{\theta} & Inn(Q)  \\ 
H  \arrow{ru}{\beta} & 
\end{tikzcd} .$$ 
Let $\rho$ be an irreducible representation of $Conj(G)$ over a vector space $V$. $\rho_{Inn(Q)}\circ \beta$ is an irreducible projective representation of $H$. Since $G$ is a Schur cover of $H$, there is a linear irreducible group representation $\rho':G\to GL(V)$ such that for $g\in G$:
$$\overline{\rho'(g)}=(\rho_{Inn(Q)}\circ\beta)\circ\alpha = \rho_{Inn(Q)}\circ \theta (g).$$
Combining the last equation with $2)$ of proposition \ref{pro}, we get for $g\in G$ :
$$\rho'_{Inn(Q)}(L_g)=\overline{\rho'(g)}=\rho_{Inn(Q)}\circ\theta(g)=\rho_{Inn(Q)}(L_g).$$
Hence, $\rho'_{Inn(Q)}=\rho_{Inn(Q)}$. It follows by $1)$ of proposition \ref{pro1} that $\rho=\chi \cdot \rho'$ for some character $\chi $ of $Conj(G)$. This proves the theorem.
\end{proof}
\begin{remark}
Considering the determinant one shows that the representations $\chi \cdot \rho'$ of the therorem take values in the unitary group if and only if $\rho'$ is unitary and $\chi$ takes values in complex numbers with modulus $1$.
\end{remark}
\begin{example}
The groups $Q_{4n}$ and $D_{4n}$ for $n$ even are Schur covers of $D_{2n}$. Hence, the irreducible quandle representation of $Conj(Q_{4n})$ $($resp. $Conj(D_{4n})$$)$ for $n$ even can be constructed out of the irreducible representations of $Q_{4n}$ $($resp. $D_{4n}$$)$.
\end{example}
\begin{remark}\label{rem}
We have associated to an irreducible quandle represenation a projective representation of the inner automorphism group. The proof of the last theorem shows that under the assumption of the theorem all projective irreducible representations of the inner automorphism group lift to quandle representations. We also note that in the case of the previous exemple the inner automorphism group $D_ {2n}$ $($$n$ even$)$ admits projective irreducibile representations that can't be lifted to linear representations of $D_{2n}$. Hence by $3)$ of proposition \ref{pro} we have projective irreducible representations of $D_{2n}$ giving quandle representation that can't be equivalent to those inducing linearisble representations of $D_{2n}$.
\end{remark}
\section{Finite quotients of envoloping groups}
In this section $Q$ is a \textbf{finite} quandle and representations are representations \textbf{over} $\mathbb{C}$. We have seen that we have a commutative diagramm : 
 
$$\begin{tikzcd}
G(Q) \arrow{r}{\theta_{G(Q)}} &Inn(Q)  \\ 
Q \arrow{u}{\varphi_Q} \arrow{ru}{\theta} & 
\end{tikzcd} ,$$ 
where $\theta$ is a quandle morphism that maps an element of $Q$ to its corresponding left traslation. We have also seen that the kernel of $\theta_{G(Q)}$ lies in the center of $G(Q)$ and we have denoted it $Z_0$. Since $Q$ is finite the group $G(Q)$ is finitely generated. $Z_0$ is a finite index subgroup. Hence $Z_0$ is finitely generated. Since $Z_0$ is abelian, we have an isomorphism :
$$\alpha : Tor(Z_0)\oplus Z_0/Tor(Z_0)\to Z_0,$$
where $Tor(Z_0)$ is the torsion subgroup (finite) of $Z_0$. The following proposition is well known :
\begin{proposition}\label{ab}
\item[1)] The abelianisation $G(Q)^{ab}$ of $G(Q)$ is isomorphic to the free abelian group over the set of orbits of $Q$ under the action of $Inn(Q)$. 
\item[2)] The abelianisation map with the identification in $1)$ maps a generator of $G(Q)$ into its class in $Q/Inn(Q)$.
\end{proposition}
\begin{proposition}\label{ab1}
\item[1)] The group $Tor(Z_0)$ is a subgroup of the derived group of $G(Q)$.
\item[2)] The natural map from $\alpha(Z_0/Tor(Z_0))$ into the abelianisation of $G(Q)$ is injective.
\end{proposition}
\begin{proof}
$1)$ follows from the fact that $G(Q)^{ab}$ is a free abelian group. We prove $2)$. $Z_0$ has finite index in $G(Q)$ hence the center of $G(Q)$ is of finite index. It follows form Schur's theorem for the derived group that the derived group of $G(Q)$ is finite. This proves $2)$.
\end{proof}
\begin{definition}
The representation group $G(Q)_\alpha$ of $Q$ assosiated to $\alpha$ is the quotient : $$G(Q)/ \alpha( Z_0/Tor(Z_0)).$$
\end{definition}
Denote by $q_\alpha$ the quotient map $G(Q)\to G(Q)_\alpha$. The group morphism $\theta_{G(Q)}: G(Q)\to Inn(Q)$ factors throught $q_\alpha$ and we get a commutative diagram :
$$\begin{tikzcd}
G(Q)_\alpha \arrow{rd}{\overline{\theta}_{G(Q)}} & \\
G(Q) \arrow{u}{q_\alpha}\arrow{r}{\theta_{G(Q)}} &Inn(Q)  \\ 
Q \arrow{u}{\varphi_Q} \arrow{ru}{\theta} & 
\end{tikzcd} ,$$ 

\begin{proposition}
\item[1)] The extension $\overline{\theta}_{G(Q)}: G(Q)_\alpha \to Inn(Q)$ is a stem extension with nucleus isomorphic to the finite group $Tor(Z_0)$.
\item[2)] The group $G(Q)_\alpha$ is finite.
\end{proposition}
\begin{proof}
$1)$ follows from the definition of $G(Q)_\alpha$ and $1)$ of proposition \ref{ab1}. $2)$ is a consequence of $1)$ of this proposition since $Tor(Z_0)$ and $Inn(Q)$ are both finite.
\end{proof}
For a multiplicative character $\chi$ of a group $G$ and a representation $\rho :G \to GL(V)$, we define the linear representation $\chi\cdot \rho : G \to GL(V)$ by:
$$(\chi \cdot \rho) (g)= \chi(g)\rho(g),$$
for $g\in G$. 

\begin{proposition}\label{tc}
Let $\rho : G(Q) \to GL(V)$ be an irreducible representation of the group $G(Q)$. There exists a multiplicative character $\chi$ of $G(Q)$ such that the kerenel of $\chi \cdot \rho$ contains $\alpha(Z_0/Tor(Z_0))$. In particular, the representation $\chi \cdot \rho$ induces an irreducible representation $\rho_\chi : G(Q)_\alpha \to GL(V)$ and $\rho=\chi^{-1}\cdot (\rho_\chi\circ q_\alpha)$, where $q_\alpha :G(Q)\to G(Q)_\alpha$ is the quotient map. 
\end{proposition}
\begin{proof}
$\alpha(Z_0/Tor(Z_0))$ lies in the center of $G(Q)$ and hence acts by scalars in the representation. Let $\lambda : \alpha(Z_0/Tro(Z_0))\to \mathbb{C}^\times$ be the associated multiplicative character. By $2)$ of proposition \ref{ab1} the natural map $\alpha(Z_0/Tor(Z_0))\to G(Q)^{ab}$ is injective and by $1)$ of proposition \ref{ab} $G(Q)^{ab}$ is a finitely generated free abelian group. Hence, the character $\lambda$ extends to a multiplicative character $\tilde{\lambda}:G(Q)\to \mathbb{C}^\times$. We obtain the proposition by taking $\chi=\tilde{\lambda}^{-1}$.
\end{proof}
We will denote by $\varphi_{Q,\alpha}$ the composition of $\varphi_Q:Q\to G(Q)$ with the quotient map $q_\alpha:G(Q)\to G(Q)_\alpha$.
\begin{theorem}
The irreducible quandle representations of $Q$ are the representations of the form $\chi \cdot( \rho \circ \varphi_{Q,\alpha})$, where $\chi$ is a character of $Q$ and $\rho$ is an irreducible representation of the finite group $G(Q)_\alpha$.
\end{theorem}
\begin{proof}
The fact that the representations of this for are irreducible follows from the fact that the image of $\varphi_{Q,\alpha}$ generates $G(Q)_\alpha$. Now, let $\rho$ be an irreducible represenation of $Q$. By the universal property of $G(Q)$ $\rho$ induces an irreducible group represetation of $\rho_{G(Q)}$ such that $\rho=\rho_{G(Q)}\circ \varphi_Q$. By the previous proposition $\rho_{G(Q)}=\lambda^{-1}\cdot (\rho'\circ q_\alpha)$ for some mutiplicative character of $G(Q)$ and an irreducible representation of $G(Q)$. Combining the last two equations we get :
$$\rho=(\lambda^{-1}\circ \varphi_Q)\cdot(\rho'\circ \varphi_{Q,\alpha}).$$
This proves that $\rho$ is of the form as in the theorem since $\lambda^{-1}\circ \varphi_Q$ is a character of $Q$. We have proved the theroem.
 \end{proof}
\begin{remark}
Considering the determinant one shows that the representations $\chi \cdot (\rho\circ \varphi_{Q,\alpha})$ of the therorem take values in the unitary group if and only if $\rho'$ is unitary and $\chi$ take values in complex numbers with modulus $1$.
\end{remark}
We have proved in the previous section that an irreducible quandle representation $\rho$ of $Q$ induces an irreducible projective representation of $\rho_{Inn(Q)}$ of $Inn(Q)$.
Proposition \ref{tc} shows that irreducible projective representations of $Inn(Q)$ that lift to linear representations of $G(Q)_\alpha$ (via $\overline{\theta}_G(Q)$) are exactly those that lift to $G(Q)$ (via $\theta_Q$). This proves that projective irreducible representations of $Inn(Q)$ that are induced by irreducible quandle representations of $Q$ are those that lift to $G(Q)_\alpha$. Now $G(Q)_\alpha$ is a stem extension of $Inn(Q)$, hence its nucleus $q_\alpha(Tor(Z_0))\simeq Tor(Z_0)$ correspond (by the inflation-restriction exact sequence with coefficients in $\mathbb{C}^\times$) to a subgroup $M(Q)_\alpha$ of the Schur multiplier $H^2(Inn(Q),\mathbb{C^\times})$. The classes defined by $M(Q)_\alpha$ are the classes of irreducible projective representations of $Inn(Q)$ that lift to linear represenations of $G(Q)_\alpha$ and hence these are the classes of irreducible representations induced by irreducible representations of $Q$. In particular, $M(Q)_\alpha\simeq Tor(Z_0)$ does not depend on $\alpha$ and allows to determine wich projective irreducible representations of $Inn(Q)$ are induced by quandle irreducible representations. Hence, we have the following :

\begin{theorem}
A finite quandle $Q$ defines a subgroup $M_Q$ of $H^2(Inn(Q),\mathbb{C}^\times)$. This subgroup is isomorphic to $Tor(Z_0)$ and correpond to the classes of irreducible projective representations of $Inn(Q)$ that are induced by irreducible quandle representations of $Q$. 
\end{theorem}
\begin{corollary}\label{cor}
If $H^2(Inn(Q),\mathbb{C}^\times)=0$ then $Tor(Z_0)=0$.
\end{corollary}
\begin{corollary}
If $H^2(Inn(Q),\mathbb{C}^\times)=0$ then $G(Q)_\alpha$ is isomorphic to $Inn(Q)$.
\end{corollary}
\begin{proposition}
Let $G$ $($finite$)$ be a shcur Cover of another group. For $Q=Conj(G)$, $Tor(Z_0)$ is the torsion group of the center and is isomorphic to $M_Q$ and $H^2(Inn(G),\mathbb{C}^{\times})$.
\end{proposition}
\begin{proof}
Follows from the previous theorem, example \ref{ex} and remark \ref{rem}, since every class of $H^2(Inn(G),\mathbb{C}^{\times})$ correspond to at least one irreducible projective representation.
\end{proof}

\begin{corollary}
If $G$ is a Schur cover of another group, then $G(Q)_\alpha$ is a Schur cover of $Inn(Q)$.
\end{corollary}

Combining example \ref{ex}, corollary \ref{cor} and the last proposition one can compute the torsion $ Tor(Z(G(Q)))$ of the center $Z(G(Q))$ of $G(Q)$, for $Q=Conj(Q_{4n})$ and $Q=Conj(D_{2n})$, from Schur multipliers and obtain the following table :

$$\begin{tabular}{ c|c|c|c|c} 

Q & Inn(Q) & $H^2(Inn(Q),\mathbb{C}^\times)$& $M_Q$ & Tor(Z(G(Q))) \\
\hline
$Conj(Q_{4n})$ $n$ odd & $D_{2n}$ & 0& 0 & 0\\ 
$Conj(Q_{4n})$ $n$ even  & $D_{2n}$ & $\mathbb{Z}/2\mathbb{Z}$ & $\mathbb{Z}/2\mathbb{Z}$ & $\mathbb{Z}/2\mathbb{Z}$\\ 
$Conj(D_{2n})$ $n$ odd & $D_{2n}$ & 0& 0 & 0\\ 
$Conj(D_{4n})$ $n$ odd & $D_{2n}$  & 0 & 0 &0 \\
$Conj(D_{4n})$ $n$ even  & $D_{2n}$ & $\mathbb{Z}/2\mathbb{Z}$ &$\mathbb{Z}/2\mathbb{Z}$  & $\mathbb{Z}/2\mathbb{Z}$\\

\end{tabular}$$
\\
According to \cite{leb}, for $Q=Conj(S_n)$, where $S_n$ is the symmetric group, the center of $G(Q)$ has no torsion. Hence $M_Q\simeq Tor(Z_0)$ is trivial, while $H^2(Inn(Q),\mathbb{C}^\times)$ is not for $n\geq 4$, $($$Inn(Q)\simeq S_n$ for $n\geq 3$$)$. This means that the only irreducible projective representations of $Inn(Q) $ induced by quandle irreducible representions are linearisable. It follows from proposition \ref{pro1} that the irreducible quandle representations of $Con(S_n)$ are the representations of the form $\chi \cdot \rho $ where $\chi$ is a quandle character and $\rho$ is an irreducible linear group representation of $S_n$.

\begin{remark}
For $Q=Conj(G)$, replacing the central extension $G(Q)\to Inn(Q)$ , by the central extension $G(Q) \to G$. One can definie as for $Inn(Q)$ a finite quotient of $G(Q)$ that is a stem extension of $G$. One shows that irreducible quandle representations of $Conj(G)$ can be constructed out of irreducible representations of the quotient and characters. One can also prove that $Conj(G)$ determines a subgroup of the Schur multiplier of $G$ corresponding to projective irreducible representation of $G$ that are induced by quandle representations of $Conj(G)$.
\end{remark}


\begin{thebibliography}{test}

\bibitem[Ka87]{Sc}
Karpilovsky, Gregory, \newblock The Schur Multiplier.\newblock Clarendon Press, 1987.
\bibitem[EM14]{EM}
Eisermann, Michael, \newblock Quandle Coverings and Their Galois Correspondence. \newblock Fundamenta Mathematicae, vol. 225, no. 1, pp. 103–67,  2014.
\bibitem[ME18]{rep} 
Elhamdadi, M. and Moutuou, E. kaïoum M.. \newblock Finitely stable racks and rack representations. \newblock  Communications in Algebra, 46(11), 4787–4802, 2018.
\bibitem[Dar22]{dar}
Jacques Darné, \newblock Nilpotent quandles. \newblock arxiv, 2022
\bibitem[LE24]{leb}
Lebed Victoria, Conjugation groups and structure groups of quandles, arxiv.
\bibitem[NDCV26]{str}
Rodr\'{i}guez-Nieto, Jos\'{e} Gregorio and Salazar-D\'{i}az, Olga Patricia and Vallejos-Cifuentes, Ricardo Esteban and Vel\'{a}squez, Ra\'{u}l, \newblock Rack representations and connections with groups representations. \newblock, Journal of Algebra and Its Applications, 2026.
\bibitem[MM26]{GQ}
Maassarani Mohamad, \newblock Groups and quandles. \newblock arxiv, 2026.



\end{thebibliography}
 \end{document}